\documentclass[12pt]{article}
\usepackage[all]{xy}
\usepackage{amsfonts,amsmath,oldgerm,amssymb,amscd,}
\usepackage{hyperref}
\newtheorem{theorem}{Theorem}[section]
\newtheorem{proposition}[theorem]{Proposition}
\newtheorem{lemma}[theorem]{Lemma}
\newtheorem{definition}[theorem]{Definition}
\newtheorem{corollary}[theorem]{Corollary}
\newtheorem{question}[theorem]{Question}

\newtheorem{problem}[theorem]{Problem}

\begin{document} 
\begin{center}
{\Large \bf Polynomial Rings Over Commutative Reduced Hopfian Local Rings }\\
\vspace{.1in} {\large Alpesh M. Dhorajia and Himadri Mukherjee }\\
\vspace{.05in} {Birla Institute of Technology and Science Pilani, India\\
$\{$\mbox{alpesh, himadrim}$\}$@goa.bits-pilani.ac.in}
\end{center}
{\small
\noindent{\bf Abstract:} In this paper we prove that if $R$ is a commutative, reduced, local ring, then $R$ is 
Hopfian if and only if the ring $R[x]$ is Hopfian. This answers a question of Varadarajan \ref{p}, in the case when $R$ is a
reduced local ring. We provide examples of non-Noetherian Hopfian commutative domains by proving that the 
finite dimensional domains are Hopfian. Also, we derive some general results related to Hopfian rings. 

\vspace*{.1in}
 \vspace{.1in}

\section{Introduction}
All the rings considered in this paper are commutative with identity. A $Hopfian$ object is an object $A$ such that 
any epimorphism of $A$ onto $A$ is necessarily an automorphism. 
The dual notion is that of a $co$-$Hopfian$ object, which is an object $B$ such that every monomorphism from 
$B$ into $B$ is necessarily an automorphism. The notion of Hopfian groups was introduces by Baumslag in \cite{Ba}, the notion of Hopfian group.
Hiremath, introduced the concept of Hopfian rings and Hopfian modules and in \cite{Hiremath}, 
he proved the following: Let $R$ be a boolean ring and $X$ be a space of all 
the maximal ideals of $R$ equipped with hull-kernel topology, if $R$ is a Hopfian then $X$ is a co-Hopfian, in the sense that 
every injective continuous map from $X$ into $X$ is homeomorphism.
The notion of Hopfian and co-Hopfian have been studied in the categories of 
groups, rings, modules and topological spaces. Since then the question has been generalized not only
to other categories but also has been weakened and strengthened in quest of finding a classification 
by many researchers (\cite{ASG}, \cite{Gho}, \cite{Hag}, \cite{Strong_Hopfian}, \cite{Hopfian_module}, \cite{Note}). 

In \cite{V1}, Varadarajan has studies quite intensively the Hopfian and co-Hopfian objects and he asked several interesting
questions which are still open. He asked if the analogue of Hilbert's basis theorem is valid for Hopficity of 
polynomial ring. More precisely, he asked the following:
\begin{problem}\label{p}
Let $R$ be a commutative ring, whether 
$R$ Hopfian, implies the polynomial ring $R[x]$ Hopfian?
\end{problem}
In \cite{Varadrajan}, Varadarajan
answered the above question positively in the case if $R$ is a boolean ring. More precisely, he proved that 
if $R$ is a boolean Hopfian ring then the polynomial ring $R[x]$ is also Hopfian. In this paper, 
we extend Varadarajan's result by proving: If $R$ is a commutative reduced clean ring, then 
$R$ is Hopfian implies that the polynomial ring $R[x]$ is also Hopfian.
In \cite{Tripathi}, Tripathi shows that 
if $R$ is a ring such that $a^n=a$ for all $a\in R$, then $R$ is Hopfian if and only if $R[x]$ is Hopfian, where $n$ is a
some fixed positive integer $> 1$. Further, in \cite{Note}, Tripathi and Zvengrowski generalized the above result to the ring 
$R$ in which every element $a$ satisfies $a^{n_a}=a$ for some positive integer $n_a>1$ depending on $a$. They have also shown 
that there are infinite class of examples of Hopfian rings that are non Noetherian rings. We explore some more examples of 
commutative non-Noetherian Hopfian domains by proving 
if $R$ is a commutative domain of finite dimensional then $R$ is Hopfian.

A commutative ring $R$ is said to be $clean$ 
if every element of $R$ can be written as a sum of a unit and an idempotent. The notion of a clean ring was 
introduced by Nicholson (\cite{NK}) and has been studied extensively since then by many 
researchers (see \cite{And}, \cite{Camillo}, \cite{cam}, \cite{NK1}). We prove that if $R$ is a commutative reduced clean ring,
then $R$ Hopfian implies $R[x]$ Hopfian. As a corollary we derive that for a reduced commutative local ring $R$, $R$ is Hopfian if and only if
$R[x]$ is Hopfian. It is well known that the boolean ring is a clean. 

In section $2$, we show that if $R$ a finite dimensional integral domain with unity
then it is Hopfian. As an application, we get an example of a Hopfian commutative non-Noetherian integral domain. 
This example settles query of Varadarajan (see \cite{Varadrajan}) and will extend the known classes of Hopfian rings. In section
$3$, we prove the results regarding the $R$ Hopfian vs $R[x]$ Hopfian question, which
generalize the results of \cite{Hiremath, Note, Varadrajan}. More precisely, we prove that if $R$ is a commutative clean 
Hopfian ring, then the polynomial ring $R[x]$ is also Hopfian (see \ref{theorem}).

\section{Hopficity of a General Ring}
Throughout this paper we assume that our ring $R$ is commutative and contains a non-zero identity. We prove the following observation which is crucial ingredient for the main result of this section. 
\begin{proposition}\label{p1}
Let $R$ be a commutative domain and $\phi : R\rightarrow R$ is an onto ring homomorphism. 
Let $\mathfrak{A}_0=ker \phi$ and $\mathfrak{A}_n=\phi^{-1}(\mathfrak{A}_{n-1})$ for $n\geq 1$. Then

$(1)$ $\mathfrak{A}_0\subseteq\mathfrak{A}_1\subseteq\ldots \subseteq \mathfrak{A}_n\subseteq\ldots$.

$(2)$ If $\mathfrak{A}_n=\mathfrak{A}_{n+1}$ for some $n$, then $\mathfrak{A}_0=0$.
\end{proposition}
\begin{proof}
We prove the result by using induction on $n$. Since $R$ is a domain $\mathfrak{A}_0$ is a prime ideal. 
First we show that $\mathfrak{A}_0\subseteq \mathfrak{A}_1$. Let $r\in \mathfrak{A}_0$. Since $\mathfrak{A}_0=ker \phi$, 
$\phi(r)=0$. Clearly, $\phi(r)\in \mathfrak{A}_0$ and hence $r\in \phi^{-1}(\mathfrak{A}_0)$. Therefore 
$\mathfrak{A}_0\subseteq \mathfrak{A}_1$. Suppose $\mathfrak{A}_{n-1}\subseteq\mathfrak{A}_n$ for some $n\geq 1$. We prove that 
$\mathfrak{A}_n\subseteq\mathfrak{A}_{n+1}$. Let $s\in \mathfrak{A}_n$. Since $\mathfrak{A}_n=\phi^{-1}(\mathfrak{A}_{n-1})$,
$\phi(s)\in \mathfrak{A}_{n-1}$. As $\mathfrak{A}_{n-1}\subseteq\mathfrak{A}_n$, we have $\phi(s)\in \mathfrak{A}_n$.
Therefore $s\in \phi^{-1}(\mathfrak{A}_n)=\mathfrak{A}_{n+1}$. Hence $\mathfrak{A}_n\subseteq\mathfrak{A}_{n+1}$.

We show that if $\mathfrak{A}_n=\mathfrak{A}_{n+1}$ for some $n$, then $\mathfrak{A}_0=0$. As $\phi$
is a surjective ring homomorphism and $\mathfrak{A}_n=\mathfrak{A}_{n+1}$, we have $\mathfrak{A}_{n-1}=\mathfrak{A}_{n}$. By 
induction hypothesis, $\mathfrak{A}_0=\mathfrak{A}_1$. Now to show $\mathfrak{A}_0=0$. Let $r\in \mathfrak{A}_0$. 
Since $\mathfrak{A}_0=\mathfrak{A}_1$, we have $\mathfrak{A}_0=\phi^{-1}(\mathfrak{A}_0)$. 
Therefore $\phi^{-1}(r)=s\in \mathfrak{A}_0$. Hence $\phi(s)=r$ and as $s\in \mathfrak{A}_0$, $r=0$. Therefore $\mathfrak{A}_0=0$.
$\hfill\square$
\end{proof}

It is easy to see that the Noetherian rings are Hopfian. In the following theorem we extend this known class of Hopfian rings to include the domains which are finite dimensional non-Noetherian.

\begin{theorem}\label{2}
If $R$ is a finite dimensional commutative integral domain then $R$ is Hopfian. 
\end{theorem}
\begin{proof}
Let $\phi$ be a surjective ring homomorphism from $R$ onto $R$. To show $\phi$ is an automorphism. Let $\mathfrak{A}_0=ker \phi$ 
and $\mathfrak{A}_n=\phi^{-1}(\mathfrak{A}_{n-1})$ for $n\geq 1$. By \ref{p1} $(1)$, 
$\mathfrak{A}_n$'s are prime ideals for all $n$.
Since $\phi$ is surjective ring homomorphism, we have increasing chain of prime ideals 
$\mathfrak{A}_0\subseteq\mathfrak{A}_1\subseteq\ldots \subseteq \mathfrak{A}_n\subseteq\ldots$ of $R$. Since $R$ is 
of finite dimensional, there exists a positive integer $n\geq 1$ such that $\mathfrak{A}_{n-1}=\mathfrak{A}_{n}$. By \ref{p1} $(2)$,
$\mathfrak{A}_0=0$, i.e. $Ker\phi =0$. Hence $\phi$ is an automorphism.
\end{proof}

\begin{remark}
 Note that the statement of the above theorem can be generalized to include the rings which are infinite dimensional but with no infinite strictly increasing ascending 
 chain of prime ideals. We have included a question at the end of this section in this regard.
\end{remark}

\begin{remark}
In above result, the hypothesis integral domain and finite dimension conditions are essential, as we have the following examples.
\end{remark}

\begin{example}
Let $A=K[x_1,x_2,\ldots]$ be the polynomial ring with infinitely many variables and let $I=(x_1^2,x_2^2,\ldots)$ be an ideal of $A$, where $K$ is a field. Let $R=A/I$ be 
a quotient ring. Then $dimR=0$. Let $\phi$ be a ring homomorphism from $R$ onto $R$ defined by 
$\phi(x_1)=0$ and $\phi(x_{i+1})=x_{i}$ for $i\geq 1$. 
Clearly, $\phi$ is an onto ring homomorphism from $R$ to $R$ but not an automorphism. Therefore $R$ is not Hopfian.
\end{example}

\begin{example}
Let $R=K[x_1,x_2,\ldots]$ be the polynomial ring with infinitely many variables, where $K$ is a field. Clearly $R$ is not finite dimensional. Define $\phi$ a ring homomorphism
from $R$ onto $R$ by $\phi(x_1)=0$ and $\phi(x_{i+1})=x_{i}$ for $i\geq 1$. Clearly $\phi$ is an onto ring homomorphism, but not injective, hence
not an automorphism. Therefore $R$ is not Hopfian.
\end{example}

We have the following example of finite dimensional commutative non-Noetherian Hopfian domain, which gives positive answer to the 
query of K. Varadarajan (see \cite{V1}). Note that it is easy to see that all the commutative Noetherian rings are Hopfian rings.

\begin{example}
By \ref{2}, Nagata's example of finite dimensional non Noetherian domain is Hopfian. 
\end{example}
\begin{proposition}
Let $R$ be a commutative Hopfian domain. Suppose $\phi: R[x]\rightarrow R[x]$ is surjective ring homomorphism.
If $\phi(R)\subseteq R$ then $\phi$ is an automorphism.
\end{proposition}
\begin{proof}
First, we show that the restriction map $\phi$ from $R$ into $R$ is also surjective. Let $r\in R$ and 
suppose $f(x)\in R[x]$ such that $\phi(f(x))=r$. We show that $f(x)\in R$. 
Let $\phi(x)=b_0+b_1x+\cdots + b_mx^m$, where $m\geq 0$. Let $f(x)=a_0+a_1x+\cdots + a_nx^n$, where $n\geq 1$.
Since $\phi(f(x))=r$, we have $\phi(a_0)+\phi(a_1)\phi(x)+\cdots+\phi(a_n)\phi(x)^n=r$. Since $\phi$ is an onto ring homomorphism
from $R[x]$ to $R[x]$, $m\geq 1$. Since $R$ is a domain, $deg\phi(x)^n >deg \phi(x)^{n-1}$, 
and hence $\phi(a_i)=0$ for $1\leq i\leq n$. Therefore $\phi(a_0)=r$. Hence restriction of $\phi$ on $R$ is an onto ring 
homomorphism from $R$ onto $R$. Since $R$ is a Hopfian ring, the restriction of $\phi$ over $R$ is an automorphism from
$R$ to $R$. 

To show $\phi$ is an automorphism, it is enough to show $\phi(x)=ax+b$, where $a, b\in R$ such that $a$ is unit in $R$. 
Let $g(x)=c_0+c_1x+\cdots+c_dx^d\in R[x]$, where $d\geq 1$ such that $\phi(g(x))=x$. 
Therefore $\phi(c_0)+\phi(c_1)\phi(x)+\cdots+\phi(c_d)\phi(x)^d =x$. Since $R$ is domain, again we have $\phi(c_i)=0$ 
for $2\leq i\leq d$. Since $\phi(x)=b_0+b_1x+\cdots + b_mx^m$ and $\phi(c_0)+\phi(c_1)\phi(x)=x$, $b_i=0$ for $2\leq i\leq m$.
Since $\phi(c_0)+\phi(c_1)(b_0+b_1x)=x$, we have $b_1$ is unit and hence $\phi(x)=b_0+b_1x$. Therefore $\phi$ is an automorphism
from $R[x]$ to $R[x]$.
\end{proof}
\begin{corollary}
Let $R$ be a commutative Hopfian integral domain. Suppose $\phi: R[x]\rightarrow R[x]$ is surjective ring homomorphism.
If there exists an automorphism $\psi: R[x]\rightarrow R[x]$ such that $\psi\circ \phi(R)\subset R$ then $\phi$ is an automorphism. 
\end{corollary}

In the light of the theorem \ref{2} we would like to ask the following question.

\begin{question}
 Does there exist a Hopfian commutative integral domain that contains a strictly increasing infinite chain of prime ideals?
\end{question}

\section{The Hopficity of $R[x]$}
Our main goal here is to prove that for a commutative reduced local ring $R$, $R[x]$ is Hopfian if and only if $R$ is. We will prove which by showing the same result for a general
class of rings namely the clean rings. Note that this will also generalize Varadarajan's result(\cite{Varadrajan}, Theorem 2).
\begin{theorem}\label{varad_main}
Let $R$ be a boolean ring. If $R$ is Hopfian, then $R[x]$ is a Hopfian ring. 
\end{theorem}

 Throughout this section we assume that the ring $R$ is 
always commutative, reduced and with identity. 
The main theorem of this section is the following:

\begin{theorem}\label{main}
Let $R$ be a commutative reduced clean ring. If $R$ is Hopfian, then $R[x]$ is Hopfian ring. 
\end{theorem}

\begin{definition}
An element $x \in R$ is called clean if it may be written as the
sum of a unit and an idempotent. If every element of $R$ is clean then we say that
$R$ is a clean ring.
\end{definition}
In proving \ref{main}, we follow the terminology and techniques developed by Varadarajan in (\cite{Varadrajan}). 
Let $R$ denote a commutative ring with identity and $I=\{p(x)\in R[x]|p(0)=0\}$. Then $R[x]=R\oplus I$. 
We denote any element of $R[x]$ as a
column vector 

\begin{equation*}
\begin{bmatrix}
           r \\
           p(x) \\
\end{bmatrix}, ~~\mbox{where}~~ r\in R ~~   \mbox{and}~~  p(x)\in I.
\end{equation*} 
Then the multiplication in $R[x]$ is given by the following formula:
 \begin{align}
\begin{bmatrix}
           r \\
           p(x) \\
\end{bmatrix}
\begin{bmatrix}
           s \\
           q(x) \\
\end{bmatrix}=
\begin{bmatrix}
           rs \\
           rq(x)+sp(x)+p(x)q(x)\\
\end{bmatrix}.
\end{align} 
 
Any $\mathbb{Z}$-homomorphism $f:R[x]\rightarrow R[x]$ can be represented by a matrix, i.e. 
\begin{equation*}
f=\begin{bmatrix}
           \varphi_{11}~~ \varphi_{12}\\
           \varphi_{21}~~ \varphi_{22} \\
\end{bmatrix},
\end{equation*} 
where $\varphi_{11}\in Hom_{\mathbb{Z}}(R,R)$, $\varphi_{12}\in Hom_{\mathbb{Z}}(I,R)$, $\varphi_{21}\in Hom_{\mathbb{Z}}(R,I)$ and
$\varphi_{22}\in Hom_{\mathbb{Z}}(I,I)$. For any $r\in R$ and $p(x)\in I$, we have 
\begin{equation*}f
\begin{bmatrix}
           r \\
           p(x) \\
\end{bmatrix} =
\begin{bmatrix}
           \varphi_{11}(r)+\varphi_{12}(p(x)) \\
            \varphi_{21}(r)+\varphi_{22}(p(x))\\
\end{bmatrix}.
\end{equation*} 

In view of this we have the following lemma.

\begin{lemma}\label{lemma1}
Let $R$ be a commutative reduced clean ring. If $f:R[x]\rightarrow R[x]$ is a ring homomorphism, then $\varphi_{21}=0$. 
\end{lemma}
\begin{proof}
It is enough to show that $f(R)\subseteq R$. By contradiction, suppose for some $r\in R$, $f(r)= b_0+b_1x+b_2x^2+\cdots + b_nx^n$, 
where $b_i\in R$ for $0\leq i\leq n$ and $b_n\neq 0$ for some $n\geq 1$. Since $R$ is a clean ring $r=u+e$, where $u$ is a unit and $e$ is an 
idempotent in $R$. As $f$ is a ring homomorphism, $f(r)=f(u+e)=f(u)+f(e)$. Since $u\in R$ is a unit, $f(u)$ is a unit in $R[x]$.
Since $R$ is reduced, by (\cite{AT}, Ex. 2, p-10), we have $f(u)\in R$. Therefore the degree of $f(e)$ is $n\geq 1$. Since
$e\in R$ is an idempotent, $e^2=e$, we have $f(e^2)=f(e)$. Since $f$ is a ring homomorphism $f(e)^2=f(e)$. Since $R$ is 
a reduced, the degree of $f(e)^2$ is $2n$, where $n\geq 1$. Since $f(e)=f(e)^2$, a contradiction and hence 
the degree of $f(e)$ is zero. Therefore $f(e)\in R$.
Hence $f(r)=f(u)+f(e)\in R$. $\hfill\square$
\end{proof}
\begin{remark}
In view of the above lemma, $f$ has the following expression:

\begin{equation*}
f=\begin{bmatrix}
           \varphi_{11}~~ \varphi_{12}\\
           ~0~~~~ \varphi_{22} \\
\end{bmatrix}.
\end{equation*}
\end{remark}

\begin{lemma}\label{lemma3}
Let $R$ be a commutative reduced clean ring and 
\begin{equation*}
f=\begin{bmatrix}
           \varphi_{11}~~ \varphi_{12}\\
           ~0~~~~ \varphi_{22} \\
\end{bmatrix}:R[x]\rightarrow R[x]
\end{equation*} 
be a surjective ring homomorphism.
Then for every positive integer $n$, $\varphi_{22}(x^n)=\varphi_{22}(x)q(x)$, where $q(x)\in R[x]$. 
\end{lemma}
\begin{proof}
We prove by induction on $n$. It is trivial for $n=1$. Assume $n>1$. Suppose 
$\varphi_{22}(x^{n-1})= \varphi_{22}(x)q'(x)$ for some $q'(x)\in R[x]$. Now consider the following expression:
\begin{align*}
\begin{bmatrix}
           0\\
           x^{n} \\
\end{bmatrix}=
\begin{bmatrix}
           0 \\
           x \\
\end{bmatrix}
\begin{bmatrix}
           0 \\
           x^{n-1}\\
\end{bmatrix}.
\end{align*}
Applying $f$ (i.e. matrix expression of $f$) on both sides of the above expression, we get 
\begin{equation*}
 \varphi_{22}(x^n)= \varphi_{12}(x)\varphi_{22}(x^{n-1})+
\varphi_{12}(x^{n-1})\varphi_{22}(x)+\varphi_{22}(x)\varphi_{22}(x^{n-1}).
\end{equation*}
Substituting the value of $\varphi_{22}(x^{n-1})$ using the induction hypothesis, the result follows. $\hfill\square$
\end{proof}
\begin{proposition}\label{proposition}
Let $R$ be a commutative reduced clean ring and 
\begin{equation*}
f=\begin{bmatrix}
           \varphi_{11}~~ \varphi_{12}\\
           ~0~~~~ \varphi_{22} \\
\end{bmatrix}:R[x]\rightarrow R[x]
\end{equation*} 
be a surjective ring homomorphism.
Then $\varphi_{22}(x)=ux$, where $u\in R$ is a unit  and $\varphi_{11}:R\rightarrow R$ is a surjective ring homomorphism.
\end{proposition}
\begin{proof}
Since $\varphi_{22}:I\rightarrow I$, suppose $\varphi_{22}(x)=a_1x+a_2x^2+\cdots+a_nx^n$, where $n\geq 1$.
Since $f$ is a surjective ring homomorphism,
it is easy to see that the $\mathbb{Z}$-homomorphism $\varphi_{22}:I\rightarrow I$ is onto. Hence there is a polynomial
$q(x)=b_1x+b_2x^2+\cdots+b_mx^m$, such that $\varphi_{22}(q(x))=x$. By \ref{lemma3}, $\varphi_{22}(q(x))=\varphi_{22}(x)q'(x)$, 
where $q'(x)\in R[x]$. Therefore $(a_1x+a_2x^2+\cdots+a_nx^n)q'(x)=x$. Hence $x((a_1+a_2x+\cdots a_nx^{n-1})q'(x)-1)=0$, i.e.
$(a_1+a_2x+\cdots a_nx^{n-1})q'(x)-1=0$. Therefore $a_1+a_2x+\cdots a_nx^{n-1}$ is unit in $R[x]$. Since $R$ is reduced and by
(\cite{AT}, Ex. 2, p-10), $a_1$ is a unit and $a_i=0$ for $2\leq i\leq n$, i.e. $\varphi_{22}(x)=a_1x$, where $a_1$ is a unit in $R$.

To show $\varphi_{11}:R\rightarrow R$ is a surjective ring homomorphism. In equation (1), by setting 
$p(x)=0=q(x)$ and applying $f$ on both sides, we get $\varphi_{11}(rs)=\varphi_{11}(r)\varphi_{11}(s)$ for all $r,s\in R$.
Now, to show $\varphi_{11}$ 
is onto. Let $r\in R$. By the above paragraph, there exists a unit $u \in R$ such that $\varphi_{22}(x)=ux$. 
Since the $\mathbb{Z}$-homomorphism $\varphi_{22}$ is surjective, there is a polynomial $p(x)\in I$ such that 
$\varphi_{22}(p(x))=rux$. Suppose $p(x)=a_1x+a_2x^2+\cdots+a_kx^k$, where $a_i\in R$ and $a_k\neq 0$ for some $k\geq 1$. Therefore, we get 
\begin{equation*}
\varphi_{11}(a_1)\varphi_{22}(x)+\varphi_{11}(a_2)\varphi_{22}(x^2)\cdots+\varphi_{11}(a_k)\varphi_{22}(x^k)=rux. 
\end{equation*}
Since $\varphi_{22}(x^j)$ is a polynomial of degree $j$ with leading coefficient a unit for $1\leq j\leq k$, we have 
$\varphi_{11}(a_j)
\varphi_{22}(x^j)=0$ for $2\leq j\leq k$. Therefore $\varphi_{11}(a_1)\varphi_{22}(x)=rux$. Since 
$\varphi_{22}(x)=ux$, $\varphi_{11}(a_1)ux=rux$. Hence $\varphi_{11}(a_1)=r$, i.e. $\varphi_{11}$ is onto. $\hfill\square$
\end{proof}
\begin{theorem}\label{theorem}
Let $R$ be a commutative reduced clean Hopfian ring. Then $R[x]$ is a Hopfian ring. 
\end{theorem}
\begin{proof}
Let $f:R[x]\rightarrow R[x]$ be any surjective ring homomorphism. Write $R[x]=R\oplus I$. By \ref{lemma1}, $f$ has the following expression:  
\begin{equation*}
f=\begin{bmatrix}
           \varphi_{11}~~ \varphi_{12}\\
           ~0~~~~ \varphi_{22} \\
\end{bmatrix},
\end{equation*}
where $\varphi_{11}\in Hom_{\mathbb{Z}}(R,R)$, $\varphi_{12}\in Hom_{\mathbb{Z}}(I,R)$ and
$\varphi_{22}\in Hom_{\mathbb{Z}}(I,I)$. By \ref{proposition}, $\varphi_{11}:R\rightarrow R$ is a surjective ring homomorphism and 
$\varphi_{22}(x)=ux$, where $u\in R$ is a unit. Since $R$ is Hopfian, it follows that $\varphi_{11}:R\rightarrow R$ 
is an automorphism.
\begin{align*}
\mbox{Let}
\begin{bmatrix}
           r \\
           p(x) \\
\end{bmatrix}\in \mbox {Ker} f. ~~~\mbox{Then} ~~\varphi_{22}(p(x))=0.~~\mbox{Hence, it is enough to show that}~~ p(x)=0.
\end{align*}
Suppose $p(x)\neq 0$, then $p(x)= a_1x+a_2x^2+\cdots +a_nx^n$ with $a_i\in R$ and $a_n\neq 0$ for some $n\geq 1$. We know that
$\varphi_{22}(x^j)=u_jx^j$ for $1\leq j\leq n$, where $u_j\in R$ is unit for all $1\leq j\leq n$. Therefore $\varphi_{22}(p(x))$
is a polynomial with leading coefficient $\varphi_{11}(a_n)u_n$. Since $\varphi_{11}:R\rightarrow R$ is an isomorphism, we 
see that $\varphi_{11}(a_n)\neq 0$ and hence $\varphi_{11}(a_n)u_n\neq 0$. Thus $\varphi_{22}(p(x))=0$ implies that $p(x)=0$. Also
since $\varphi_{11}$ is an isomorphism, $\varphi_{11}(r)=0$ implies $r=0$. Hence $f:R[x]\rightarrow R[x]$ is an isomorphism. 
$\hfill\square$
\end{proof}

Note that a boolean ring is a clean ring which is reduced. As a result we derive \ref{varad_main} as a corollary.
The following result is well known, we give the proof for sake of completeness. 
\begin{lemma}\label{lemma4}
If $R$ is a commutative local ring, then $R$ is clean ring. 
\end{lemma}
\begin{proof}
Let $x\in R$. Let $m$ be a maximal ideal. If $x\notin m$, then $x = x + 0$, where $x$ is a unit and $0$ is an idempotent.
If $x\in m$, then $x = -1+x + 1$, where $-1+x$ is a unit and $1$ is an idempotent in $R$.  $\hfill\square$
\end{proof}
\begin{remark}\label{rem}
It is easy to see that if $R[x]$ is Hopfian, then $R$ is Hopfian. 
\end{remark}

\begin{theorem}\label{main2}
Let $R$ be a commutative reduced local ring, $R$ is Hopfian if and only if the polynomial ring $R[x]$ is Hopfian.
\end{theorem}
\begin{proof}
By \ref{theorem}, \ref{lemma4} and \ref{rem}, the proof follows. 
\end{proof}

\bibliographystyle{amsplain}
{}
\end{document}